\documentclass[12pt]{article}
\usepackage{graphicx,amsmath,amsfonts,amssymb,amscd,amsthm, mathrsfs}
\newtheoremstyle{kai}
{3pt}{3pt}{}{}{\bfseries}{.}{.5em}{}
\makeatletter
\def\EquationsBySection{\def\theequation
{\thesection.\arabic{equation}}%
\@addtoreset{equation}{section}}

\newcommand\old[1]{}
\newcommand{\pend}{\hfill \thicklines \framebox(6.6,6.6)[l]{}}
\renewenvironment{proof}{\noindent {\it  Proof.} \rm}{\pend}
\newtheorem{theorem}{Theorem}[section]
\newtheorem{lemma}{Lemma}[section]

\newtheorem{remark}{Remark}[section]

\newtheorem{example}{Example}[section]

\EquationsBySection
\makeatother

\begin{document}
\pagestyle{plain}
\title
{\bf Existence of Invariant Measures of Stochastic\\ Systems with Delay in the Highest\\ Order Partial Derivatives}
\author{
Kai Liu
\\
\\
\small{Department of Mathematical Sciences,}\\
\small{School of Physical Sciences,}\\
\small{The University of Liverpool,}\\
\small{Peach Street, Liverpool, L69 7ZL, U.K.}\\
\small{E-mail: k.liu@liv.ac.uk}\\}

\date{}
\maketitle

\noindent {\bf Abstract:} In this note, we shall consider the existence of invariant measures for a class of infinite dimensional stochastic functional differential equations with delay whose driving semigroup  is eventually norm continuous. The results obtained are applied to stochastic heat equations with distributed delays which  appear in such terms having the highest order partial derivatives. In the systems, the associated driving semigroups are generally non  eventually compact. 

\vskip 50pt
\noindent {\bf Keyword:} Invariant measure; Eventually norm continuous; Distributed delay; Stochastic functional differential equation.
\vskip 10pt

\noindent{\bf 2000 Mathematics Subject Classification(s):} 60H15, 60G15, 60H05.

\newpage
\section{Introduction}
Let $X$ and $K$ be two separable Hilbert spaces and ${\mathscr L}(K, X)$ the space of all bounded and linear operators from $K$ into $X$. We also denote by ${\mathscr L}_{HS}(K, X)$  the space of all Hilbert-Schmidt operators from $K$ into $X$. 
The goal of this paper is to consider the existence of invariant measures for the following stochastic functional differential  equation (SFDE) of retarded type on $X$, 
\begin{equation}
\label{13/09/13(1)}
\begin{cases}
\displaystyle du(t) =\Lambda u(t)dt  + \displaystyle\int^0_{-r} \alpha(\theta)\Lambda_1 u(t+\theta)d\theta dt + \Lambda_2 u(t-r)dt +\Psi(u_t)dt + \Sigma(u_t)dW(t),\,\,t\ge 0,\\
u(0)=\phi_0,\,\,\,\,u(\theta)=\phi_1(\theta),\,\,\,\theta\in [-r, 0],\,\,\,r>0,
\end{cases}
\end{equation}  
where $\Lambda: {\mathscr D}(\Lambda)\subset X\to X$ generates an analytic semigroup of bounded linear operators $e^{t\Lambda}$, $t\ge 0$, on $X$, $u_t(\theta) := u(t+\theta)$, $\theta\in [-r, 0]$, $t\ge 0$, $\alpha\in L^2([-r, 0]; {\mathbb R})$ and $\Lambda_i: {\mathscr D}(\Lambda_i)\subset X\to X$ are  linear, generally unbounded,  operators on $X$ with ${\mathscr D}(\Lambda)\subset {\mathscr D}(\Lambda_i)$, $i=1,\,2$. Here ${\mathscr D}(\Lambda)$ is the domain of $\Lambda$ which is a Hilbert space under the usual graph norm. The initial value $\phi=(\phi_0, \phi_1)$ belongs to the product space ${\cal X} = M\times L^2([-r, 0]; {\mathscr D}(\Lambda))$ with $M=({\mathscr D}(\Lambda), X)_{1/2, 2}$ being the Lions' interpolation Hilbert space between ${\mathscr D}(\Lambda)$ and $X$, $W$ is a $K$-valued Wiener process and $\Psi: L^2([-r, 0]; {\mathscr D}(\Lambda))\to M$, $\Sigma: L^2([-r, 0]; {\mathscr D}(\Lambda))\to {\mathscr L}_{HS}(K, M)$ are appropriate nonlinear Lipschitz continuous mappings. 

One of the most effective approaches in handling (\ref{13/09/13(1)}) is to lift  it up  to obtain a stochastic evolution equation without delay on ${\cal X}$. Precisely, we define a linear operator ${A}$ on ${\cal X}$ by 
\begin{equation}
\label{13/09/13(2)}
\begin{split}
{\mathscr D}({A}) = \Big\{\phi=(\phi_0, \phi_1):\, &\phi_0=\phi_1(0)\in {\mathscr D}(\Lambda),\,\,\phi_1\in W^{1, 2}([-r, 0]; {\mathscr D}(\Lambda)),\\
&\hskip 30pt \Lambda\phi_0   +  \displaystyle\int^0_{-r}\alpha(\theta)\Lambda_1\phi_1(\theta)d\theta + \Lambda_2 \phi_1(-r)\in M\Big\},
\end{split}
\end{equation}
and for any $\phi=(\phi_0, \phi_1)\in {\mathscr D}({A})$, let
\begin{equation}
\label{20/08/2013(20)}
{A}\phi= \Big(\Lambda\phi_0  +  \displaystyle\int^0_{-r}\alpha(\theta)\Lambda_1\phi_1(\theta)d\theta + \Lambda_2 \phi_1(-r), \frac{d\phi_1(\theta)}{d\theta}\Big).
\end{equation}
 It may be shown that ${A}$ generates a strongly continuous or $C_0$-semigroup $e^{t{A}}$, $t\ge 0$, on the space ${\cal X}$. As a consequence, we may lift up (\ref{13/09/13(1)})  onto ${\cal X}$ to consider a stochastic differential equation without delay on ${\cal X},$
 \begin{equation}
\label{13/09/13(3)}
\begin{cases}
\displaystyle dU(t) ={A}U(t)dt + F(U(t))dt + {B}(U(t))dW(t),\,\,\,\,t\ge 0,\\
U(0)=(\phi_0, \phi_1)\in {\cal X},
\end{cases}
\end{equation}  
where $U(t)= (u(t), u_t)$,  $t\ge 0$ and $F: {\cal X}\to {\cal X}$ and ${B}: {\cal X}\to {\mathscr L}(K, {\cal X})$ are defined respectively by
\[
 {F}:\, (\phi_0, \phi_1)\to (\Psi\phi_1, 0)\hskip 15pt\hbox{for any}\hskip 15pt (\phi_0, \phi_1)\in {\cal X},
\]
and
\[
{B}:\, (\phi_0, \phi_1)\to (\Sigma\phi_1, 0)\hskip 15pt\hbox{for any}\hskip 15pt (\phi_0, \phi_1)\in {\cal X}.\]
For instance, the equation (\ref{13/09/13(1)}) can be applied to a class of stochastic partial integrodifferential equations with distributed delays in the highest order partial derivatives of the form
\begin{equation}
\label{17/12/13(1)}
\begin{cases}
 du(t, x) =\Delta u(t, x)dt  + \displaystyle\int^0_{-r} \alpha(\theta)\Delta u(t+\theta, x)d\theta dt + \Sigma(u(t-r, x)) dW(t, x),\\
\hskip 250pt\,\,\,\,(t, x)\in [0, \infty)\times {\cal O},\\
u(0)=\phi_0,\,\,\,\,u(\theta)=\phi_1(\theta),\,\,\,\theta\in [-r, 0],\,\,\,r>0,
\end{cases}
\end{equation}
where $\Delta$ is the usual Laplacian operator, ${\cal O}$ is an open bounded subset of ${\mathbb R}^N$ with regular boundary $\partial{\cal O}$ and $\Sigma:\, {\mathbb R}\to  {\mathscr L}(K, W^{1, 2}({\mathbb R}^N))$ is an appropriate nonlinear Lipschitz function.
 
In this work, we are concerned about  the existence of invariant measures for the type of equation (\ref{13/09/13(3)}), which involves, in essence, a compactivity argument (Krylov-Bogoliubov theory). In finite dimensional spaces, the compactivity problem could be reduced to the investigation of boundedness for a solution of (\ref{13/09/13(3)}). However, the space ${\cal X}$ is  infinite dimensional when $r>0$, and due to the absence of local compactness of ${\cal X}$, we need to exploit compact properties of the stochastic differential equation (\ref{13/09/13(3)}). The reader is referred to, e.g., \cite{gdpjz96} for a systematic statement about invariant measures of stochastic systems. In contrast with non time delay systems,  the treatment here is complicated due to the infinite dimensional nature of delay problems. For instance, if both $\Lambda_1$ and $\Lambda_2$ in (\ref{13/09/13(1)}) are bounded on $X$ and $\Lambda$ generates a compact semigroup $e^{t\Lambda}$ on $t> 0$,  it may be shown (see, e.g., \cite{kl09(2)}) that ${A}$ in (\ref{13/09/13(3)}) generates an eventually compact semigroup $e^{t{A}}$, i.e., $e^{t{A}}$   is compact for all $t>r$. In this case, the existence of an invariant measure of (\ref{13/09/13(3)}) is considered in \cite{jbovgsl09} when some further conditions on the diffusion term are imposed. When $\Lambda_1$ or $\Lambda_2$ in (\ref{13/09/13(1)}) is unbounded on $X$, it is generally untrue that  the associated semigroup  $e^{t{A}}$, $t\ge 0$, of the system (\ref{13/09/13(3)}) is eventually compact.  For example, it is shown in 
\cite{Gdbkkes85(2)} that the associated semigroup  $e^{t{A}}$, $t\ge 0$, of the system (\ref{17/12/13(1)}) is never compact or eventually compact, 
a fact which means the theory in  \cite{jbovgsl09} cannot be applied to a system like (\ref{17/12/13(1)}).  However, it is known in  \cite{jj1991} that the semigroup  $e^{t{A}}$, $t\ge 0$, associated with (\ref{17/12/13(1)})
is eventually norm continuous, i.e., $t\to e^{t{A}}$ is continuous with respect to the operator norm $\|\cdot\|$ on all $t\ge t_0$ for some $t_0>0$.  

In this work, we will study the existence of invariant measures to Eq.  (\ref{13/09/13(3)}) where  ${A}$ generates an eventually norm continuous semigroup on ${\cal X}$. As a consequence, the theory established in the note is applied to a system like (\ref{17/12/13(1)}) to get an invariant measure. The author is also referred to \cite{es84} for a concrete example of noncompact but norm comtinuous semigroup on a Hilbert space.

\section{Main Results}

With some abuse of notation, we shall consider throughout the remainder of the work the mild solutions of the following stochastic evolution equation on a Hilbert space $H$,
   \begin{equation}
\label{18/12/13(1)}
\begin{cases}
\displaystyle dy(t) = Ay(t)dt + F(y(t))dt + B(y(t))dW(t),\,\,\,\,t\ge 0,\\
y(0)=y_0\in H,
\end{cases}
\end{equation}  
where $A$ generates a $C_0$-semigroup $e^{tA}$, $t\ge 0$, on $H$, $W$ is a $K$-valued cylindrical Wiener process and $F: H\to H$, $B:\, H\to {\mathscr L}_{HS}(K, H)$ are two globally Lipschitz mappings.

\begin{remark}\rm
We cannot establish existence of invariant measures by exploiting dissipativity property of the drift parts of the equation (cf. Ch. 6 in \cite{gdpjz96}) since this condition does not hold in our situation. Indeed, let us consider a real stochastic differential equation with point delay,
   \begin{equation}
\label{18/12/13(160)}
\begin{cases}
\displaystyle dy(t) = by(t-r)dt + cy(t)dB(t),\,\,\,\,t\ge 0,\\
y(t)=0, \,\,\,t\in [-r, 0],
\end{cases}
\end{equation} 
where $b,\,c\in {\mathbb R}$ and $B$ is a standard one-dimensional Brownian motion. Let ${\cal X} = {\mathbb R}\times L^2([-r, 0]; {\mathbb R})$. For  arbitrarily given  $a\ge 0$, we have
\[
\begin{split}
\langle (F-aI)(\phi) - (&F-aI)(\psi), \phi-\psi\rangle_{\cal X}\\
& = (\phi_0 -\psi_0)[\phi_1(-r) - 0] -a\Big[(\phi_0-\psi_0)^2 + \int^0_{-r} (\phi_1(\theta)-\psi_1(\theta))^2d\theta\Big] 
\end{split}
\]
for any $\phi,\,\psi\in {\mathbb R}\times W^{1, 2}([-r, 0]; {\mathbb R})$.
We claim that for every $a\ge 0$, there exist $\phi,\,\psi\in {\cal X}$ such that the right-hand side in the previous expression is strictly positive. Indeed, let $\psi=(0, 0)$ and $\phi_0>0$, we can make the expression $\phi_0\cdot \phi_1(-r) - a(x^2_0 + \int^0_{-r}\phi_1(\theta)^2d\theta)$ as large as possible by choosing $\phi_1(-r)$ properly and, in the meanwhile, fixing the value  $\phi_0$ and keeping $\int^0_{-r} \phi_1(\theta)^2d\theta$ unchanged.  
\end{remark}

In the sequel, we impose the following conditions on (\ref{18/12/13(1)}). 
\begin{enumerate}
\item[(H1)] The semigroup $e^{tA}$ is norm continuous on $[t_0, \infty)$ for some $t_0>0$;
\item[(H2)]
$B: H\to {\mathscr L}_{HS}(K, H)$ admits a factorization $B=LD$ such that  $D: H\to {\mathscr L}_{HS}(K, H_1)$ is globally Lipschitz where $H_1$ is a seperable Hilbert space and $L\in {\mathscr L}(H_1, H)$ with $t\to e^{tA}L$ being norm continuous on $[0, t_0]$, e.g., $L$ is compact or $L=e^{t_0A}$ with $X=H$.
\item[(H3)]
$F: H\to H$ admits a factorization $F=UV$ such that  $V: H\to H_2$ is globally Lipschitz where $H_2$ is a seperable Hilbert space and $U\in {\mathscr L}(H_2, H)$ with $t\to e^{tA}U$ being norm continuous on $[0, t_0]$.
\end{enumerate}

At the end of this work, a concrete class of stochastic delay heat equations satisfying (H1), (H2) and (H3) are shown how to use for us the theory in the work to practical problems. On this occasion, we only note that (H2) and (H3) may be satisfied when, for example,  $L$ is compact or $L=e^{t_0A}$ with $H_1=H$ for (H2) and similarly for (H3).

\begin{lemma}
\label{25/12/13(10)}
Suppose that $A$ generates a norm continuous semigroup $e^{tA}$ on $[t_0, \infty)$ for some $t_0>0$. Let
\[
B_R(0) = \{x\in H:\, \|x\|_H\le R\},\hskip 20pt R>0,\]
then $e^{t_0A}B_R(0)$ is relatively compact in $H$ for any $R>0$.
\end{lemma}
\begin{proof}
For arbitrary $R>0$ and sequence $\{x_n\}_{n\ge 1}\in B_R(0)$, let us consider the set 
\[
y_n(t) := e^{tA}x_n\in H,\hskip 15pt n\in {\mathbb R},\hskip 15pt t\in [t_0, t_0+1].\]
Since $\|e^{tA}\|\le Ce^{\mu t}$ for some constants $C>0$, $\mu\ge 0$ and all $t\ge 0$, it is easy to see that  
\[
\max_{\substack{n\ge 1,\\ t\in [t_0, t_0+1]}}\|y_n(t)\|_H = \max_{\substack{n\ge 1,\\t\in [t_0, t_0+1]}}\|e^{tA}x_n\|_H \le Ce^{\mu (t_0+1)}R<\infty.\]
On the other hand, since $e^{tA}$ is norm continuous on $[t_0, \infty)$ and $\{x_n\}\subset B_R(0)$, we have for any $s,\,t\in [t_0, t_0+1]$ that
\[
\lim_{t\to s}\max_{n\ge 1}\|e^{tA}x_n-e^{sA}x_n\|_H \le \lim_{t\to s}\|e^{tA}-e^{sA}\|R\to 0.\]
By virtue of Ascoli-Alzel\`a Theorem, we thus have that $\{y_n(t)\}$, $n\in {\mathbb N}$, is relatively compact in $C([t_0, t_0+1], H)$. That is, there exists a subsequence, still denote it by $y_n(t)$, and the corresponding $x_n\in B_R(0)$, $n\ge 1$, such that $\{y_n(t)\}=\{e^{tA}x_n\}$ is convergent in $C([t_0, t_0+1]; H)$, i.e., there exists $y(t)\in C([t_0, t_0+1]; H)$ such that 
\[
\lim_{n\to\infty}\max_{t\in [t_0, t_0+1]}\|y_n(t)-y(t)\|_H =  \lim_{n\to\infty}\max_{t\in [t_0, t_0+1]}\|e^{tA}x_n-y(t)\|_H =0,\]
which particularly implies that 
\[
\lim_{n\to\infty}\|e^{t_0 A}x_n-y(t_0)\|_H=0.\]
This means that $e^{t_0 A}B_R(0)$ is relatively compact in $H$. The proof is complete now.
\end{proof}

Let $y(t, y_0)$, $t\ge 0$,  be the mild solution of Eq. (\ref{18/12/13(1)}). Firstly, let us consider the stochastic convolution 
\[
u(t, y_0) := \int^{t}_0 e^{(t-s)A}B(y(s, y_0))dW(s),\hskip 20pt t\ge 0.\]

\begin{lemma} 
\label{26/12/13(1)}
Suppose that the conditions (H1) and (H2) hold for some $t_0>0$.
For any $\varepsilon>0$ and $R>0$, there exists  a compact set $S_{R, \varepsilon}\subset H$ such that 
\[
{\mathbb P}(u(t_0, y_0)\in S_{R, \varepsilon})>1-\varepsilon\hskip 20pt \hbox{for all}\hskip 20pt \|y_0\|_H\le R.\]
\end{lemma}
\begin{proof}
Recall the factorization $B= LD$ in (H2) through Hilbert space $X$ and let us consider the set 
\begin{equation}
\label{25/12/13(1)}
V := \{e^{sA}Lx:\, s\in [0, t_0], x\in X, \|x\|_X\le R\}\subset H,\hskip 20pt R>0.
\end{equation}
We show that $V$ is relatively compact in $H$. Indeed, let $\{v_n\}$, $n\in {\mathbb N}$, be an arbitrary sequence in $V$. Then there exist sequences $\{s_n\}\subset [0, t_0]$ and $\{x_n\}\subset X$ with $\|x_n\|_X\le R$, $n\in {\mathbb N}$, such that
\begin{equation}
\label{25/12/13(2)}
v_n = e^{s_nA}Lx_n,\hskip 20pt n\in {\mathbb N}.
\end{equation}
For the sequence $\{s_n\}_{n\ge 1}\subset [0, t_0]$, there exists a subsequence, still denote it by $\{s_n\}$, and a number $s_0\in [0, t_0]$ such that $s_n\to s_0$ as $n\to\infty$. Since $s\to e^{sA}L$ is norm continuous on $[0, t_0]$, by analogy with Lemma \ref{25/12/13(10)}, it is not difficult  to see that the sequence $\{e^{s_0A}Lx_n\}$, $n\in {\mathbb N}$, is relatively compact in $H$. Thus there exists a subsequence of $\{x_n\}$, still denote it by $\{x_n\}$, such that $e^{s_0A}Lx_n\to z$ as $n\to \infty$ for some $z\in H$. In addition to the norm continuity of $t\to e^{tA}L$ on $[0, t_0]$, this further implies that
\[
\begin{split}
\|e^{s_{n}A}Lx_{n} - z\|_H &\le \|e^{s_{n}A}Lx_{n} - e^{s_0A}Lx_n\|_H + \|e^{s_{0}A}Lx_n -z\|_H\\
&\le R\|e^{s_{n}A}L - e^{s_0 A}L\|  + \|e^{s_{0}A}Lx_n -z\|_X\\
&\to 0\hskip 15pt \hbox{as}\hskip 15pt n\to\infty.
\end{split}
\]
That is, $V$ is relatively compact. Therefore, there exists (e.g., see Ex. 3.8.13 (ii) in \cite{vib98}) a compact, self-adjoint, invertible operator $T\in {\mathscr L}(H)$ such that 
\[
V\subset T(B_1(0))\]
where $B_1(0)=\{x\in H:\, \|x\|_H\le 1\}$.  

Let $K(\lambda) = \lambda T(B_1(0))$ for any $\lambda>0$ and $T^{-1}\in {\mathscr L}(H)$ denote the bounded inverse of $T$. Then it is easy to see that
\[
u(t_0, y_0) \in K(\lambda) \iff \int^{t_0}_0 T^{-1}e^{(t_0-s)A}LD(y(s, y_0))dW(s)\in \lambda B_1(0),\] 
which, together with the fact $\|e^{tA}\|\le Ce^{\mu t}$, $C>0$, $\mu\ge 0$, for all $t\ge 0$, immediately implies that 
\[
\begin{split}
{\mathbb P}(u(t_0, y_0)\notin K(\lambda)) &\le {\mathbb P}\Big(\int^{t_0}_0 T^{-1}e^{(t_0-s)A}LD(y(s, y_0))dW(s)\notin \lambda B_1(0)\Big)\\
&\le \frac{1}{\lambda^2}{\mathbb E}\Big(\Big\|\int^{t_0}_0 T^{-1}e^{(t_0-s)A}LD(y(s, y_0))dW(s)\Big\|^2_H\Big)\\
&\le \frac{\|T^{-1}\|^2C^2e^{2\mu t_0}\|L\|^2}{\lambda^2}{\mathbb E}\Big(\int^{t_0}_0 \|D(y(s, y_0))\|^2_{HS}ds\Big)\\
&\le \frac{C_1}{\lambda^2}(1+ \|y_0\|^2_H)\to 0\hskip 15pt \hbox{as}\,\,\,\,\lambda\to\infty,
\end{split}
\]
for some $C_1=C_1(t_0)>0$. Here we use Th. 5.3.1 of \cite{gdpjz96} in the last inequality. The proof is thus complete.
\end{proof}

In an analogous manner, let us consider the integral  
\[
v(t, y_0) := \int^{t}_0 e^{(t-s)A}F(y(s, y_0))ds,\hskip 20pt t\ge 0.\]
Then it is possible to establish the following results.
\begin{lemma} 
\label{26/12/13(177)}
Suppose that the conditions (H1) and (H3) hold for some $t_0>0$.
For any $\varepsilon>0$ and $R>0$, there exists  a compact set $S'_{R, \varepsilon}\subset H$ such that 
\[
{\mathbb P}(v(t_0, y_0)\in S'_{R, \varepsilon})>1-\varepsilon\hskip 20pt \hbox{for all}\hskip 20pt \|y_0\|_H\le R.\]
\end{lemma}

Now we are in a position to state the main result in this work. 

\begin{theorem} 
\label{14/01/14(2)}
Suppose that the conditions (H1), (H2), (H3) 
 hold and for any $x\in H$ and $\varepsilon>0$, there exists $R>0$ such that for all $T>0$,
\begin{equation}
\label{26/12/13(70)}
\frac{1}{T}\int^T_0 {\mathbb P}(\|y(t, y_0)\|_H\ge R)dt < \varepsilon.
\end{equation}
Then there exists at least one invariant measure for the solution $y(t, y_0)$, $t\ge 0$, of  (\ref{18/12/13(1)}).
\end{theorem}
\begin{proof}
Note that for the mild solution $y(t, y_0)$, $t\ge 0$, of (\ref{18/12/13(1)}), it is valid that 
\begin{equation}
\label{26/12/13(10)}
y(t_0, y_0)= e^{t_0A}y_0 + \int^{t_0}_0 e^{(t_0 -s)A}F(y(s, y_0))ds +  \int^{t_0}_0 e^{(t_0 -s)A}B(y(s, y_0))dW(s).
\end{equation}
By virtue of Lemma \ref{25/12/13(10)}, for any $R>0$ there exists a compact set $K_1(R)\subset H$ such that $e^{t_0A}x\in K_1(R)$ for all $\|x\|_H\le R$. On the other hand, by Lemmas \ref{26/12/13(1)} and  \ref{26/12/13(177)} for any $R>0$ and $\varepsilon>0$, there exist compact sets $K_2(R, \varepsilon)$  and $K_3(R, \varepsilon)$ in $H$ such that 
\[
{\mathbb P}\Big(\int^{t_0}_0 e^{(t_0-s)A}B(y(s, y_0))dW(s)\in K_2(R, \varepsilon)\Big)>1-\varepsilon,\]
and 
\[
{\mathbb P}\Big(\int^{t_0}_0 e^{(t_0-s)A}F(y(s, y_0))ds\in K_3(R, \varepsilon)\Big)>1-\varepsilon.\]
Hence, we conclude from (\ref{26/12/13(10)}) that for $\|y_0\|_H<R$ and $\varepsilon>0$, there is
\begin{equation}
\label{26/12/13(20)}
{\mathbb P}\Big\{y(t_0, y_0)\in K_1(R)\cup K_2(R, \varepsilon)\cup K_3(R, \varepsilon)\Big\}\ge 1-\varepsilon.
\end{equation}

Let $t>t_0$, $K(R, \varepsilon) =  K_1(R)\cup K_2(R, \varepsilon)\cup K_3(R, \varepsilon)$ and $p_t(x, dy)$ be the Markov transition probabilities of the solution $y(t, y_0)$, $t\ge 0$, of (\ref{18/12/13(1)}). Then it is easy to get that
\[
\begin{split}
{\mathbb P}(y(t, y_0)\in K(R, \varepsilon)) &= {\mathbb E}[p_{t_0}(y(t-t_0, y_0), K(R, \varepsilon))]\\
&\ge {\mathbb E}\big[p_{t_0}(y(t-t_0, y_0), K(R, \varepsilon)){\bf 1}_{\{\|y(t-t_0, y_0)\|_H\le R\}}\big],\hskip 20pt \forall\, t>t_0,
\end{split}
\]
which, together with (\ref{26/12/13(20)}), immediately implies that 
\[
{\mathbb P}(y(t, y_0)\in K(R, \varepsilon))\ge (1-\varepsilon){\mathbb P}(\|y(t-t_0, y_0)\|_H\le R),\hskip 20pt \forall\, t>t_0.\]
This further yields that 
\begin{equation}
\label{14/01/2014(1)}
\frac{1}{T}\int^{T+t_0}_{t_0}{\mathbb P}(y(t, y_0)\in K(R, \varepsilon))dt \ge \frac{1-\varepsilon}{T}\int^T_0 {\mathbb P}(\|y(t, y_0)\|_H\le R)dt.
\end{equation}
Choosing $R>0$ large enough and $\varepsilon>0$ small enough, we have  by virtue of (\ref{26/12/13(70)}) and (\ref{14/01/2014(1)}) that the family 
\[
\frac{1}{T}\int^{T+t_0}_{t_0}p_t(y_0, \cdot)dt,\,\,\,\, T>0,\]
is tight. According to the classic Krylov-Bogoliubov theory, there exists an invariant measure. The proof is complete now.
\end{proof}

\begin{example}\rm
Consider the stochastic reaction-diffusion equation with delay
\begin{equation}
\label{14/01/14(10)}
\begin{cases}
\displaystyle\frac{\partial y(t, x)}{\partial t} = \displaystyle\frac{\partial^2}{\partial x^2}y(t, x) + \displaystyle\int^0_{-r} \alpha(\theta) \frac{\partial^2}{\partial x^2}y(t+\theta, x)d\theta + \Sigma \dot W(t, x),\hskip 20pt t\ge 0,\hskip 15pt x\in {\mathbb R},\\
y(\theta, \cdot) = \phi_1(\theta, \cdot)\in W^{1, 2}({\mathbb R}),\hskip 15pt \theta\in [-r, 0],\\
y(0, x) = \phi_0(x)\in L^2({\mathbb R}),\hskip 20pt x\in {\mathbb R},
\end{cases}
\end{equation}
where $r>0$, $\alpha\in L^2([-r, 0]; {\mathbb R})$ and $\Sigma\in {\mathscr L}_{HS}(L^2({\mathbb R}), W^{1, 2}({\mathbb R}))$. Here $\dot W(t, x)$, $t\ge 0$, $x\in {\mathbb R}$, is a cylindrical  space and time white noise.

We can formulate this equation by setting $H =K= L^2({\mathbb R})$, $M= W^{1, 2}({\mathbb R})$ and ${\cal X} = M\times L^2([-r, 0]; W^{2, 2}({\mathbb R}))$ with $\Lambda$ defined by 
\[
\begin{split}
&\Lambda = \Lambda_1 = \frac{\partial^2}{\partial x^2},\\
&{\mathscr D}(\Lambda) = \Big\{y\in W^{2,2}({\mathbb R}):\,\, \lim_{x\to \pm\infty}y(x) =0,\,\,\lim_{x\to \pm\infty}\frac{dy(x)}{dx} =0\Big\},
\end{split}
\]
and 
\[
B= (\Sigma, 0)\,\,\,\,\hbox{on}\,\,\,\,{\cal X}.\]
Then we have an abstract stochastic evolution equation without delay on ${\cal X}$,
 \begin{equation}
\label{18/12/13(102)}
\begin{cases}
\displaystyle dY(t) = AY(t)dt + BdW(t),\,\,\,\,t\ge 0,\\
Y(0)=(\phi_0, \phi_1)\in {\cal X},
\end{cases}
\end{equation}  
where $Y(t)= (y(t), y_t)$, $t\ge 0$, and $A$ generates a $C_0$-semigroup $e^{tA}$, $t\ge 0$, on ${\cal X}$. Note that the theory in \cite{jbovgsl09} cannot be applied to this example since $A$ does not generate an eventually compact semigroup (cf. \cite{Gdbkkes85(2)}). However, it is known (see \cite{mm02})
that  the associated ${A}$ does generate an eventually norm continuous semigroup $e^{tA}$, $t\ge 0$. This means that we can still use the theory in the note to this example.

Indeed, on this occasion we may write (with some abuse of notation) that $\Sigma=i\circ \Sigma$ where $i:\, W^{1, 2}({\mathbb R})\to L^2({\mathbb R})$ is the canonical embedding from $W^{1, 2}({\mathbb R})$ into $L^2({\mathbb R})$, which is known to be a compact mapping. Hence, the operator $B$ admits a  factorization satisfying the conditions in (H2). Further, we may conclude from Theorem \ref{14/01/14(2)} that if the mild solution of (\ref{14/01/14(10)}) are bounded in probability, then an invariant measure exists. 
\end{example}


\begin {thebibliography}{17}




\bibitem{absp05} A. B\'atkai and S. Piazzera. {\it Semigroups for Delay Equations.} A K Peters, Wellesley, Massachusetts, (2005).

\bibitem{jbovgsl09} J. Bierkens, O. van Gaans and S. Lunel. Existence of an invariant measure for stochastic evolutions driven by an eventually compact semigroup. {\it J. Evol. Equ.} {\bf 9}, (2009), 771--786.

\bibitem{vib98} V. Bogachev. {\it Gaussian Measures.} American Mathematical Society, (1998).

\bibitem{gdpjz96} G. Da Prato and J. Zabczyk. {\it Ergodicity for Infinite Dimensional Systems.} Cambridge University Press, (1996).


\bibitem{Gdbkkes85(2)} G. Di Blasio, K. Kunisch and E. Sinestrari. Stability for abstract linear functional differential equations. {\it Israel J. Math.} {\bf 50}, (1985), 231--263.

\bibitem{kern00} K. Engel and R. Nagel. {\it One-Parameter Semigroups for Linear Evolution Equations}. Graduate Texts in Mathematics, {\bf 194}, Springer-Verlag, New York, Berlin, (2000).

\bibitem{jj1991} J. Jeong. Stabilizability of retarded functional differential equation in Hilbert space. {\it Osaka J. Math.} {\bf 28}, (1991), 347--365.

\bibitem{kl09(2)} K. Liu. Retarded stationary Ornstein-Uhlenbeck processes driven by L\'evy noise and operator self-decomposability. {\it Potential Anal.} {\bf 33},  (2010), 291--312.

  \bibitem{jllem72} J. L. Lions and E. Magenes. 
 {\it Probl\`emes aux Limites non Homog\`enes et Applications.} Dunod, Paris, (1968).

\bibitem{mm02} M. Mastin\v{s}ek. Norm continuity and stability for a functional differential equation in Hilbert space. {\it J. Math. Anal. Appl.} {\bf 269}, (2002), 770--783.



\bibitem{es84} E. Sinestrari. A noncompact differentiable semigroup arising from an abstract delay equation. {\it C. R. Math. Rep. Acad. Sci. Canada.} {\bf 6}, (1984), 43--48.





\end{thebibliography}
\end{document}